\documentclass[preprint,12pt]{elsarticle}
\usepackage{amsfonts}

\usepackage{graphics}

\usepackage{amssymb}
\usepackage{amsthm}
\usepackage{amsfonts}
\usepackage{mathrsfs}
\usepackage{amscd,amssymb,amsmath,graphicx,verbatim,xy,bm}
\usepackage[dvips]{hyperref}
\usepackage[TS1,OT1,T1]{fontenc}
\newtheorem{theorem}{Theorem}[section]
\newtheorem{lemma}[theorem]{Lemma}
\newtheorem{corollary}[theorem]{Corollary}
\newtheorem{proposition}[theorem]{Proposition}

\theoremstyle{definition}

\newtheorem{example}[theorem]{Example}

\theoremstyle{remark}
\newtheorem{remark}[theorem]{Remark}

\catcode`@=11
\let\@int\int \def\int{\displaystyle\@int}
\let\@lim\lim \def\lim{\displaystyle\@lim}
\let\@sum\sum \def\sum{\displaystyle\@sum}
\let\@sup\sup \def\sup{\displaystyle\@sup}
\let\@inf\inf \def\inf{\displaystyle\@inf}
\let\@cap\cap \def\cap{\displaystyle\@cap}
\let\@cup\cup \def\cup{\displaystyle\@cup}
\let\@max\max \def\max{\displaystyle\@max}
\let\@min\min \def\min{\displaystyle\@min}
\let\@frac\frac \def\frac{\displaystyle\@frac}
\let\@iint\iint \def\iint{\displaystyle\@iint}

\def\epsilon{\varepsilon}

\numberwithin{equation}{section}

\begin{document}

\begin{frontmatter}
\title{Topologically conjugate classification of diagonal operators}

\author[a]{Yue Xin}
\ead{179929393@qq.com}
\author[b]{Bingzhe Hou\corref{*}}
\cortext[*]{Corresponding author.}\ead{houbz@jlu.edu.cn}
\address[a]{School of Mathematical Science,
Heilongjiang University, 150080, Harbin, P. R. China}
\address[b]{School of Mathematics, Jilin University, 130012, Changchun, P. R. China}

\begin{abstract}
Let $\ell^{p}$, $1\leq p<\infty$, be the Banach space of absolutely $p$-th power summable sequences and let $\pi_{n}$ be the natural projection to the $n$-th coordinate for $n\in\mathbb{N}$. Let $\mathfrak{W}=\{w_{n}\}_{n=1}^{\infty}$ be a bounded sequence of complex numbers. Define the
operator $D_{\mathfrak{W}}: \ell^{p}\rightarrow\ell^{p}$ by, for any $x=(x_{1},x_{2},\ldots)\in \ell^p$,
$\pi_{n}\circ D_{\mathfrak{W}}(x)=w_{n}x_{n}$ for  all $n\geq1$. We call $D_{\mathfrak{W}}$ a diagonal operator on $\ell^{p}$.
In this article, we study the topological conjugate classification of the diagonal operators on $\ell^{p}$. More precisely, we obtained the following results. $D_{\mathfrak{W}}$ and $D_{\vert\mathfrak{W}\vert}$ are topologically conjugate, where $\vert\mathfrak{W}\vert=\{\vert w_{n}\vert\}_{n=1}^{\infty}$. If $\inf_{n}\vert w_n\vert>1$, then $D_{\mathfrak{W}}$ is topologically conjugate to $2\mathbf{I}$, where $\mathbf{I}$ means the identity operator. Similarly, if $\inf_{n}\vert w_n\vert>0$ and $\sup_{n}\vert w_n\vert<1$, then $D_{\mathfrak{W}}$ is topologically conjugate to $\frac{1}{2}\mathbf{I}$. In addition, if $\inf_{n}\vert w_n\vert=1$ and $\inf_{n}\vert t_n\vert>1$, then $D_{\mathfrak{W}}$ and $D_{\mathfrak{T}}$ are not topologically conjugate.
\end{abstract}
\end{frontmatter}

\section{Introduction and preliminaries}

Let $X$ be a complete separable metric space and let $f:
X\rightarrow X$ be a continuous function. Then, the pair $(X, f)$ is called a discrete dynamical system.
In the study of dynamical systems, topologically conjugate classification is an important and difficult task.

Two dynamical systems $(X, f)$ and $(Y, g)$ are said  to be topologically conjugate, if there
exists a homeomorphism $h: X\rightarrow Y$ such that $g=h\circ
f\circ h^{-1}$. Moreover, we say that $h$ is a topological conjugacy from $f$ to $g$. Notice
that topologically conjugate relation is an equivalent relation, which preserves almost all dynamical properties such as periodicity,
transitivity, Devaney chaos and so on (refer to \cite{Devaney}).

The study of topologically conjugate classification begin with the issue \cite{S-1963} by S. Smale in the 1960s. So far, it has been achieved some results of classifications on some classes of dynamical systems. R. Adler and R. Palais \cite{AP-1965} studied the topologically conjugate classifications of the Anosov diffeomorphisms on differential manifolds, especially $T^n$. Then, the topologically conjugate classifications of the affine transformations of $T^n$ and more general compact abelian groups were studied by  R. Adler, C. Tresser and P. Worfolk \cite{ATW-1997}, P. Walters \cite{WP-1968, WP-1969}, S. Bhattacharya \cite{BS-2000, BS2-2000}.
J. Robin~\cite{R-1972} studied the topological conjugacy and structural stability for discrete dynamical systems, and then he cooperated with N. Kuiper to investigate the topologically conjugate classifications of the linear endomorphisms on finite dimensional linear spaces in \cite{KR-1973}.
R. Schult \cite{S-1977} proved that the topologically conjugate classification of the linear representations of any compact Lie group is equivalent to their algebraically conjugate classification. X. Pan and B. Hou gave the topologically conjugate classifications of the translation actions on some compact connected  Lie groups \cite{PH-2021, PH-2022}.

In operator theory, similarity is an important research object. Let $\mathcal{X}$ be a separable complex Banach space. Denote by ${B}(\mathcal{X})$ the set of all bounded linear operators from $X$ to itself. Given $A, B\in{B}(\mathcal{X})$. If there exists a bounded invertible operator $T$ such that $TA=BT$, we say that $A$ is similar to $B$. Speaking intuitively, similarity is just the linear conjugate equivalence. In recent years, linear dynamics has become a popular direction in operator theory, which is closely related to the famous invariant subspace problem. For more about linear dynamics, we refer to the books \cite{Bay} and \cite{GEP}. Now we are interested in the topologically conjugate classifications of the bounded linear operators. B.Hou, G. Liao and Y. Cao \cite{HLC12} studied the topologically conjugate classifications of a class of weighted backward shift operators. In the present paper, we focus on the topologically conjugate classifications of diagonal operators.

The class of diagonal operators is an elementary and important class operators in functional analysis and operator theory (see \cite{GMR-2003} for example), which is
crucial in the study of quantum mechanics. Let $\mathcal{H}$ be a separable complex Hilbert space and $D\in{B}(\mathcal{H})$. $D$ is called a diagonal operator, if there exists an orthonormal base $\{e_n\}_{n=1}^{\infty}$ of $\mathcal{H}$ and a bounded sequence of complex numbers $\{w_n\}_{n=1}^{\infty}$ such that
$$
Te_n=w_n e_n \ \ \ \text{for  all} \  n=1,2,\cdots.
$$
That is, the matrix representation of $D$ under the orthonormal base $\{e_n\}_{n=1}^{\infty}$ is an infinite dimensional diagonal matrix.

In the study of eigenvectors and eigenvalues for self-adjoint operators on Hilbert spaces, the diagonal operator is an elementary model and the diagonalization is a useful tool.
In 1909, H. Weyl \cite{HW-1909} proved that a self-adjoint operator in $\mathcal{B(H)}$ is a compact perturbation of a diagonal operator. T. Kato \cite{TK-1957} and M. Rosenblum \cite{MR-1957} conducted research on the perturbation of the continuous spectrum by trace class operators, they showed that if a self-adjoint operator $T$ in $\mathcal{B(H)}$ is not purely singular, then $T$ cannot be diagonalized modulo the trace class. R. Carey and J. Pincus \cite{RC-1976} proved that a purely singular self-adjoint operator in $\mathcal{B(H)}$ is a small trace class perturbation of a diagonal operator. D. Voiculescu \cite{DV-1979} proved that a normal operator in $\mathcal{B(H)}$ is a sum of a diagonal operator and a small Hilbert-Schmidt operator. More recently, Q. Li, J. Shen and R. Shi \cite{QL-2020} provided a generalized version of Voiculescu's theorem for normal operators to semifinite von Neumann algebras. Beyond the self-adjoint operators, B. Ahmadi Kakavandi and E. Nobari \cite{BA-2022} considered the diagonalization of Toeplitz operators. The diagonalization problem of bosonic quadratic Hamiltonians on Fock spaces was studied by P. Nam \cite{NP-2016}.

We could extend the concept of diagonal operator from separable complex Hilbert spaces to
separable complex Banach space with unconditional bases. Let $\mathcal{X}$ be a separable complex Banach space with unconditional bases and $D\in{B}(\mathcal{X})$. $D$ is called a diagonal operator, if there exists an unconditional base $\{e_n\}_{n=1}^{\infty}$ of $\mathcal{X}$ and a bounded sequence of complex numbers $\{w_n\}_{n=1}^{\infty}$ such that
$$
Te_n=w_n e_n \ \ \ \text{for  all} \  n=1,2,\cdots.
$$
In this paper, we aim to study the topologically conjugate classifications of the diagonal operators on the classical Banach space $\ell^{p}$.
Here, $\ell^{p}$ is the space of absolutely $p$-th power summable sequences, i.e.,
\[
\ell^{p}\triangleq\{x=(x_{1},x_{2},\ldots); x_i\in\mathbb{C} \ \text{and} \ \sum\limits_{n=1}^{\infty}\vert x_n\vert^{p}<\infty\}.
\]
We use
$\parallel x\parallel_{p}$ to denote the $\ell^p$-norm of $x\in\ell^p$, and denote by $\pi_{n}$ the natural projection to the $n$-th coordinate, i.e.,
$\pi_{n}(x)=x_{n}$ for $n\geq1$. Without confusion, we also use $0$ to denote the zero vector in $\ell^{p}$.
Let $\mathfrak{W}=\{w_{n}\}_{n=1}^{\infty}$ be a bounded sequence of complex numbers. Define the
operator $D_{\mathfrak{W}}: \ell^{p}\rightarrow\ell^{p}$ by, for any $x=(x_{1},x_{2},\ldots)\in \ell^p$,
\[
\pi_{n}\circ D_{\mathfrak{W}}(x)=w_{n}x_{n}, \ \ \text{for  all} \ n\geq1.
\]

\section{A class of homeomorphisms on $\ell^{p}$}

In this section, we aim to introduce a class of homeomorphisms on $\ell^{p}$, $1\leq p< \infty$, which plays an important role in the classification of diagonal operators.

Let $\mathfrak{S}=\{s_{n}\}_{n=1}^{\infty}$ be a bounded sequence of
positive numbers with $s_{n}\geq1$ for each $n\geq1$. Define a map
$h^{\mathfrak{S}}_{p}$ on $\ell^{p}$  by
\[
\pi_{n}\circ h^{\mathfrak{S}}_{p}(x)=\frac{x_{n}}{\vert x_{n}\vert} \cdot
\sqrt[p]{(\sum_{k=n}^{\infty}\vert x_{k}\vert^{p})^{s_{n}}-(\sum_{k=n+1}^{\infty}\vert x_{k}\vert^{p})^{s_{n}}},
\]
for any $x=(x_{1},x_{2},\ldots)\in \ell^{p}$. We will show that
$h^{\mathfrak{S}}_{p}$ is a homeomorphism from $\ell^{p}$ onto itself as the following Key Lemma.

\vspace{2mm}

\noindent \textbf{Key Lemma.} \
Let $\mathfrak{S}=\{s_{n}\}_{n=1}^{\infty}$ be a bounded sequence of
positive numbers with $s_{n}\geq1$ for each $n\geq1$. Then the map
$h^{\mathfrak{S}}_{p}$ defined as above is a homeomorphism from
$\ell^{p}$ onto itself.

\vspace{2mm}

Let us begin with two elementary results.

\begin{lemma}\label{guji}
Let $r$ be a positive integer. Then, for any $0\leq b \leq a<1$ and
any $1\leq s \leq r$, we have
\[
(1-a)(a^{r}-b^{r})\leq a^{s}-b^{s}\leq \frac{a-b}{1-a}.
\]
\end{lemma}

\begin{proof}
If $a=0$, the desired inequalities is obvious. Now suppose $a>0$. By
continuity argument, it suffices to show that the conclusion holds when $s$ is a
rational number. Let $s=\frac{m}{n}$, where $m$ and $n$ are two coprime positive integers. Denote
$u=b^{\frac{1}{n}}$ and $v=a^{\frac{1}{n}}$. Since $0\leq b \leq a<1$, we have
\[
1\geq\frac{u}{v}\geq(\frac{u}{v})^{2}\geq\cdots\geq(\frac{u}{v})^{rn-1}.
\]
Consequently,
\[
\frac{1+u+u^{2}+\cdots+u^{n-1}}{1+v+v^{2}+\cdots+v^{n-1}}
\geq \frac{1+u+u^{2}+\cdots+u^{m-1}}{1+v+v^{2}+\cdots+v^{m-1}}\geq
\frac{1+u+u^{2}+\cdots+u^{rn-1}}{1+v+v^{2}+\cdots+v^{rn-1}}.
\]
After multiplying $\frac{1-u}{1-v}$, we obtain that
\begin{equation}\label{shen}
\frac{1-b^{r}}{1-a^{r}} \leq \frac{1-b^{s}}{1-a^{s}} \leq
\frac{1-b}{1-a}.
\end{equation}

Notice that $0<a \leq a^{s} \leq a^{r}$ and
\[
a^{s}-b^{s}=(\frac{1-b^{s}}{1-a^{s}}-1)(1-a^{s}).
\]
Then, it follows from the inequalities (\ref{shen}) that
\[
(1-a)(a^{r}-b^{r})\leq a^{s}-b^{s}\leq \frac{a-b}{1-a}.
\]
\end{proof}

\begin{lemma}\label{yasuo}
Let $a>0$ and $s\geq1$. Let $f:[0,+\infty)\rightarrow[0,+\infty)$ be
the continuous map defined by, for any $x\in[0,+\infty)$,
\[
f(x)=\sqrt[s]{a+x^{s}}.
\]
Then $f$ is a contraction, i.e.,
$\vert f(x)-f(y)\vert \leq \vert x-y\vert$, for any $x,y\geq0$.
\end{lemma}

\begin{proof}
Since $a>0$ and $s\geq1$, one can see that
\[
f'(x)=x^{s-1}(a+x^{s})^{{1}/{s}-1}=\left(\frac{x}{(a+x^s)^{{1}/{s}}}  \right)^{s-1} \leq1, \ \ for \ every \ x>0 .
\]
Then $f$ is a contraction, i.e.,
$\vert f(x)-f(y)\vert \leq \vert x-y\vert$, for any $x,y\geq0$.
\end{proof}

Now, let us study the map $h^{\mathfrak{S}}_{p}$.

\begin{lemma}\label{welldefined}
$h^{\mathfrak{S}}_{p}$ is a well-defined map from
$\ell^{p}$ to itself.
\end{lemma}

\begin{proof}
It suffices to show that $h^{\mathfrak{S}}_{p}(x)$ is actually an element in $\ell^{p}$ for any $x\in\ell^{p}$. In fact, we will give both the upper and the lower estimations for the norm of $h^{\mathfrak{S}}_{p}(x)$. Let
$r\in\mathbb{N}$ be  an upper bound of the sequence $\mathfrak{S}$. Given any $x=(x_{1},x_{2},\ldots)\in
\ell^{p}$. Let $y=h^{\mathfrak{S}}_{p}(x)=(y_{1},y_{2},\ldots)$.

\noindent \textbf{Case $1$.} Suppose that $\parallel
x\parallel_{p}^{p}\leq \frac{1}{2}$. For each $n\in\mathbb{N}$,
\[
\vert y_{n}\vert^{p}=(\sum_{k=n}^{\infty}\vert x_{k}\vert^{p})^{s_{n}}-(\sum_{k=n+1}^{\infty}\vert x_{k}\vert^{p})^{s_{n}}.
\]
By Lemma \ref{guji}, we have
\[
\frac{1}{2}\left((\sum_{k=n}^{\infty}\vert x_{k}\vert^{p})^{r}-(\sum_{k=n+1}^{\infty}\vert x_{k}\vert^{p})^{r}\right)
\leq \vert y_{n}\vert^{p} \leq 2\vert x_{n}\vert^{p}.
\]
Then,
\begin{equation}\label{guji1}
2^{-{1}/{p}}\parallel x\parallel_{p}^{r}\leq\parallel
h^{\mathfrak{S}}_{p}(x)\parallel_{p}\leq 2^{{1}/{p}}\parallel
x\parallel_{p}.
\end{equation}

\noindent \textbf{Case $2$.}  Suppose that $\parallel
x\parallel_{p}^{p}> \frac{1}{2}$. For any positive number $t$ no less than
$1$, it is obvious that $t\leq t^{s_{n}}\leq t^{r}$. Notice that for any $n\in\mathbb{N}$,
\[
\pi_{n}\circ h^{\mathfrak{S}}_{p}(tx)=t^{s_{n}}\cdot\pi_{n}\circ h^{\mathfrak{S}}_{p}(x).
\]
Since $2^{{1}/{p}}\parallel x\parallel_{p}>1$ in the present case, we have
\[
2^{{1}/{p}}\parallel x\parallel_{p}\cdot\parallel h^{\mathfrak{S}}_{p}(\frac{x}{2^{{1}/{p}}\parallel x\parallel_{p}})\parallel_{p}
\leq \parallel h^{\mathfrak{S}}_{p}(x)\parallel_{p} \leq
(2^{{1}/{p}}\parallel x\parallel_{p})^{r}\cdot\parallel
h^{\mathfrak{S}}_{p}(\frac{x}{2^{{1}/{p}}\parallel
x\parallel_{p}})\parallel_{p}.
\]
Then, following from the inequalities (\ref{guji1}), we obtain that
\begin{equation}\label{guji2}
2^{-{r}/{p}}\parallel x\parallel_{p} \leq \parallel
h^{\mathfrak{S}}_{p}(x)\parallel_{p} \leq 2^{{r}/{p}}\parallel
x\parallel_{p}^{r}.
\end{equation}

Therefore, $h^{\mathfrak{S}}_{p}$ is a
map from $\ell^{p}$ to $\ell^{p}$. The proof is finished.
\end{proof}

\begin{lemma}\label{continuous}
$h^{\mathfrak{S}}_{p}: \ell^{p}\rightarrow\ell^{p}$ is a continuous map.
\end{lemma}
\begin{proof}
Let $\{x^{(m)}\}_{m=1}^{\infty}$ be a Cauchy sequence in $\ell^{p}$ and
let $y^{(m)}=h^{\mathfrak{S}}_{p}{(x^{(m)})}$ for each $m\geq1$.
Write $x^{(m)}=(x^{(m)}_{1},x^{(m)}_{2},\ldots)$ and
$y^{(m)}=(y^{(m)}_{1},y^{(m)}_{2},\ldots)$, for each $m\geq1$. By
the construction of $h^{\mathfrak{S}}_{p}$, it is obvious that
${\{h^{\mathfrak{S}}_{p}{(x^{(m)})}\}}_{m=1}^{\infty}$ converges by
coordinates, i.e., $\{y_{n}^{(m)}\}_{m=1}^{\infty}$ is a Cauchy
sequence for each $n\in\mathbb{N}$. For any
$0<\epsilon\leq\frac{1}{2}$, there is a positive integer $N_{0}$
such that
\[
\sum\limits_{n=N_{0}}^{\infty}{\vert x^{(m)}_{n}\vert^{p}}<2^{-{1}/{p}}\cdot\epsilon,
\ \  \text{for each} \ m\geq1.
\]
According to the right inequality of
(\ref{guji1}), we have
\[
\sum\limits_{n=N_{0}}^{\infty}{\vert y^{(m)}_{n}\vert^{p}}<\epsilon,  \ \ \text{for  each} \ m\geq1.
\]
Therefore, $h^{\mathfrak{S}}_{p}$ is continuous.
\end{proof}

\begin{lemma}\label{CtoC}
Let $\{x^{(m)}\}_{m=1}^{\infty}$ be a
sequence of vectors in $\ell^{p}$. Denote
$y^{(m)}=h^{\mathfrak{S}}_{p}(x^{(m)})$ for every $m\in\mathbb{N}$. If $\{y^{(m)}\}_{m=1}^{\infty}$ is a Cauchy
sequence, then so is the sequence $\{x^{(m)}\}_{m=1}^{\infty}$.
\end{lemma}

\begin{proof}
For any $0<\epsilon\leq\frac{1}{2}$, there is a positive integer
$N_{0}$ such that
\[
\sum\limits_{n=N_{0}}^{\infty}{\vert y^{(m)}_{n}\vert^{p}}<2^{-{1}/{p}}\cdot\epsilon^{r},
 \ \ \text{for  each} \ m\geq1.
\]
According to the left inequality of
(\ref{guji1}), we have
\[
\sum\limits_{n=N_{0}}^{\infty}{\vert x^{(m)}_{n}\vert^{p}}<\epsilon,  \ \ \text{for each} \ m\geq1.
\]
Then, to prove that $\{x^{(m)}\}_{m=1}^{\infty}$ is a Cauchy sequence, it suffices to prove $\{x^{(m)}\}_{m=1}^{\infty}$ converges by
coordinates. In fact,  we should only prove that $\parallel
x^{(m)}\parallel_{p}^{p}$ converges  as $m\rightarrow\infty$. Suppose that $\parallel
x^{(m)}\parallel_{p}^{p}$ converges as $m\rightarrow\infty$. Together with the convergence of $\{y^{(m)}_1\}_{m=1}^{\infty}$,
\[
(\sum\limits_{k=2}^{\infty}{\vert x^{(m)}_{k}\vert^{p}})^{s_1}=\parallel
x^{(m)}\parallel_{p}^{ps_1}-\vert y^{(m)}_1\vert^{p}
\]
also converges and consequently $\vert x^{(m)}_1\vert$ converges. Furthermore, by the convergence of $\{\sum\limits_{k=2}^{\infty}{\vert x^{(m)}_{k}\vert^{p}}\}_{m=1}^{\infty}$ and the convergence of $\{y^{(m)}_2\}_{m=1}^{\infty}$,
\[
(\sum\limits_{k=3}^{\infty}{\vert x^{(m)}_{k}\vert^{p}})^{s_2}=(\sum\limits_{k=2}^{\infty}{\vert x^{(m)}_{k}\vert^{p}})^{s_2}-\vert y^{(m)}_2\vert^{p}
\]
also converges and consequently $\vert x^{(m)}_2\vert$ converges. Following this way, we could obtain that for each $n\in\mathbb{N}$, $\vert x^{(m)}_n\vert$ converges as $m\rightarrow\infty$.
Notice that $\frac{x^{(m)}_{n}}{\vert x^{(m)}_{n}\vert}=\frac{y^{(m)}_{n}}{\vert y^{(m)}_{n}\vert}$, for every $m,n\geq1$. Then, for each $n\in\mathbb{N}$, $x^{(m)}_n$ converges as $m\rightarrow\infty$.

Now, let us show  that $\{\parallel x^{(m)}\parallel_{p}^{p}\}_{m=1}^{\infty}$ is a Cauchy sequence.
For any $\delta>0$, we have known that there is $N_{1}\in\mathbb{N}$ such that
\[
\sum\limits_{n=N_{1}+1}^{\infty}{\vert x^{(m)}_{n}\vert ^{p}}<\frac{\delta}{3},
\ \text{for  each} \ m\geq1.
\]
Since $\{y^{(m)}\}_{m=1}^{\infty}$ is a
Cauchy sequence and $1\leq s_{n}\leq r$ for all $n\in\mathbb{N}$, there
exists $M\in\mathbb{N}$ such that, for any $m_{1},m_{2}>M$,
\[
\sum\limits_{n=1}^{N_{1}}{\left\vert \vert(y^{(m_{1})}_{n})\vert^{{p}/{s_{n}}}-
\vert(y^{(m_{2})}_{n})\vert^{{p}/{s_{n}}}\right\vert}<\frac{\delta}{3}.
\]

Following from the construction of $h^{\mathfrak{S}}_{p}$, for any
$m\in\mathbb{N}$ and any $n\in\mathbb{N}$, we have
\[
\sum\limits_{k=n}^{\infty}{\vert x^{(m)}_{k}\vert^{p}}=
\sqrt[s_{n}]{\vert y^{(m)}_{n}\vert^{p}+
(\sum\limits_{k=n+1}^{\infty}{\vert x^{(m)}_{k}\vert^{p}})^{s_n}}.
\]

Then by Lemma \ref{yasuo}, for any $M_{1},M_{2}\in\mathbb{N}$ and
any $n\in\mathbb{N}$ ,
\begin{align}\label{trans}
&\left\vert \sum\limits_{k=n}^{\infty}{\vert x^{(M_{1})}_{k}\vert^{p}}
-\sum\limits_{k=n}^{\infty}{\vert x^{(M_{2})}_{k}\vert^{p}}\right\vert \nonumber \\
\leq&\left\vert \sqrt[s_{n}]{(\vert y^{(M_1)}_{n}\vert^{{p}/{s_n}})^{s_n}+(\sum\limits_{k=n+1}^{\infty}{\vert x^{(M_1)}_{k}\vert^{p}})^{s_n}}
- \sqrt[s_{n}]{(\vert y^{(M_2)}_{n}\vert^{{p}/{s_n}})^{s_n}+(\sum\limits_{k=n+1}^{\infty}{\vert x^{(M_1)}_{k}\vert^{p}})^{s_n}}\right\vert \nonumber \\
&  +\left\vert  \sqrt[s_{n}]{\vert y^{(M_2)}_{n}\vert^{p}+(\sum\limits_{k=n+1}^{\infty}{\vert x^{(M_1)}_{k}\vert^{p}})^{s_n}}
-\sqrt[s_{n}]{\vert y^{(M_2)}_{n}\vert^{p}+(\sum\limits_{k=n+1}^{\infty}{\vert x^{(M_2)}_{k}\vert^{p}})^{s_n}}\right\vert \nonumber \\
\leq& \left\vert \vert (y^{(M_{1})}_{n})\vert^{{p}/{s_{n}}}-
\vert (y^{(M_{2})}_{n})\vert^{{p}/{s_{n}}}\right\vert +\left\vert \sum\limits_{k=n+1}^{\infty}{\vert x^{(M_{1})}_{k}\vert^{p}}
-\sum\limits_{k=n+1}^{\infty}{\vert x^{(M_{2})}_{k}\vert^{p}}\right\vert.
\end{align}

Therefore, for any $m_{1},m_{2}>M$,
\begin{align*}
&\left\vert \parallel x^{(m_{1})}\parallel_{p}^{p}-\parallel x^{(m_{2})}\parallel_{p}^{p}\right\vert  \\
=& \left\vert \sum\limits_{k=1}^{\infty}{\vert x^{(m_{1})}_{k}\vert^{p}} -\sum\limits_{k=1}^{\infty}{\vert x^{(m_{2})}_{k}\vert^{p}}\right\vert \\
\leq& \left\vert\vert(y^{(m_{1})}_{1})\vert^{{p}/{s_1}}-
\vert(y^{(m_{2})}_{1})\vert^{{p}/{s_1}}\right\vert+\left\vert\sum\limits_{k=2}^{\infty}{\vert x^{(m_{1})}_{k}\vert^{p}}-\sum\limits_{k=2}^{\infty}{\vert x^{(m_{2})}_{k}\vert^{p}}\right\vert \\
& \ \ \cdots \cdots \ \ (\text{Applying the inequality} \ (\ref{trans}) \ N_1-1 \ \text{times}) \\
\leq& \sum\limits_{n=1}^{N_{1}}\left\vert\vert(y^{(m_{1})}_{n})\vert^{{p}/{s_n}}-
\vert(y^{(m_{2})}_{n})\vert^{{p}/{s_n}}\right\vert+\left\vert\sum\limits_{k=N_{1}+1}^{\infty}{\vert x^{(m_{1})}_{k}\vert^{p}}-\sum\limits_{k=N_{1}+1}^{\infty}{\vert x^{(m_{2})}_{k}\vert^{p}}\right\vert \\
\leq& \sum\limits_{n=1}^{N_{1}}\left\vert\vert(y^{(m_{1})}_{n})\vert^{{p}/{s_n}}-\vert(y^{(m_{2})}_{n})\vert^{{p}/{s_n}}\right\vert+\sum\limits_{k=N_{1}+1}^{\infty}{\vert x^{(m_{1})}_{k}\vert^{p}}+\sum\limits_{k=N_{1}+1}^{\infty}{\vert x^{(m_{2})}_{k}\vert^{p}} \\
<& \frac{\delta}{3}+\frac{\delta}{3}+\frac{\delta}{3} \\
=& \delta.
\end{align*}
So $\{\parallel x^{(m)}\parallel_{p}^{p}\}_{m=1}^{\infty}$ is a Cauchy sequence. This completes the proof.
\end{proof}

Based on the above preliminaries, we could prove the Key Lemma.

\begin{proof}[Proof of Key Lemma]
By Lemma \ref{welldefined} and Lemma \ref{continuous}, $h^{\mathfrak{S}}_{p}$ is a well-defined continuous map
from $\ell^{p}$ to itself.

Now we aim to show that $h^{\mathfrak{S}}_{p}$ is a bijection.
Let $E$ be the set of points in $\ell^{p}$ with
finite coordinates being nonzero. It is easy to see that $h^{\mathfrak{S}}_{p}(E)=E$. For any $y\in\ell^{p}$, there is a sequence of vectors $\{y^{(m)}\}_{m=1}^{\infty}$ in $E$ such that $y^{(m)}$ converges to $y$. Since $h^{\mathfrak{S}}_{p}(E)=E$, we have a sequence of vectors $\{x^{(m)}\}_{m=1}^{\infty}$ in $E$ such that $h^{\mathfrak{S}}_{p}(x^{(m)})=y^{(m)}$ for every $m\in\mathbb{N}$. By Lemma \ref{CtoC}, $x^{(m)}$ converges to $x$ in $\ell^{p}$. Then, it follows from the continuity of $h^{\mathfrak{S}}_{p}$ (Lemma \ref{continuous}) that $h^{\mathfrak{S}}_{p}(x)=y$. Thus, $h^{\mathfrak{S}}_{p}$ is a surjective. Suppose that there exist two distinct vectors $x, \widetilde{x}\in \ell^p$ such that $h^{\mathfrak{S}}_{p}(x)=h^{\mathfrak{S}}_{p}(\widetilde{x})=y$. Let
\[
y^{(m)}=y \ \ \text{for each} \ m\geq1,
\]
and
\[
x^{(m)}=\left\{
	\begin{array}{rcl}
		x  &    &   \text{if} \ m \ \text{is odd},\\
		\widetilde{x}    &    &   \text{if} \ m \ \text{is even}.
	\end{array}
	\right.
\]
Since the sequence $\{y^{(m)}\}_{m=1}^{\infty}$ is a Cauchy sequence and $h^{\mathfrak{S}}_{p}(x^{(m)})=y^{(m)}$ for every $m\in \mathbb{N}$, it follows from Lemma \ref{CtoC} that $\{x^{(m)}\}_{m=1}^{\infty}$ is also a Cauchy sequence. This is a contradiction. Thus, $h^{\mathfrak{S}}_{p}$ is an injective.

Applying Lemma \ref{CtoC} again, we obtain that $(h^{\mathfrak{S}}_{p})^{-1}$ is also continuous. Therefore, $h^{\mathfrak{S}}_{p}$  is a
homeomorphism from $\ell^{p}$ onto itself.
\end{proof}

\section{Topologically conjugate classification of diagonal operators}
In this section, we will discuss the topologically conjugate classification of diagonal operators on $\ell^{p}$, $1\leq p<\infty$.

\begin{lemma}\label{rotation}
Let $D_{\mathfrak{W}}$ be a bounded diagonal operator on $\ell^{p}$ with the diagonal sequence $\mathfrak{W}=\{w_{n}\}_{n=1}^{\infty}$.  Denote by $D_{\vert\mathfrak{W}\vert}$ the diagonal operator with the diagonal sequence $\vert\mathfrak{W}\vert=\{\vert w_{n}\vert\}_{n=1}^{\infty}$. If $\vert w_n\vert\neq1$ for all $n\in\mathbb{N}$, then $D_{\mathfrak{W}}$ and $D_{\vert\mathfrak{W}\vert}$ are topologically conjugate.
\end{lemma}
\begin{proof}
Given a complex number $w$ with $\vert w\vert\neq1$. If $w\neq 0$, write $w=\vert w\vert\textrm{e}^{\mathbf{i}\theta}$. We define $f_w:\mathbb{C}\rightarrow\mathbb{C}$ by $f_w(0)=0$ and for any $z\neq0$,
\[
f_w(z)=\left\{
	\begin{array}{rcl}
		z\cdot \textrm{e}^{\mathbf{i}\cdot \frac{\ln\vert z\vert}{\ln\vert w\vert} \cdot \theta}  &    &   {\text{if} \ w\neq0,}\\
		       &   &    \\
        z    &    &   {\text{if} \ w=0.}
	\end{array}
	\right.
\]
It is not difficult to see that $f_w$ is an homeomorphism on $\mathbb{C}$. In fact, $f_w^{-1}(0)=0$ and for any $z\neq0$,
\[
f_w^{-1}(z)==\left\{
	\begin{array}{rcl}
		z\cdot \textrm{e}^{-\mathbf{i}\cdot \frac{\ln\vert z\vert}{\ln\vert w\vert} \cdot \theta}      &   &   {\text{if} \ w\neq0,}\\
		       &   &    \\
        z       &   &   {\text{if} \ w=0.}
	\end{array}
	\right.
\]
Moreover, $f_w$ is a topological conjugacy from the multiplication operator $M_{\vert w \vert}$ to the multiplication operator $M_{w}$ on $\mathbb{C}$, i.e., for any $z\in\mathbb{C}$,
\begin{equation}\label{commu}
f_w(\vert w\vert \cdot z)=w\cdot f_w(z).
\end{equation}

Define $F_{\mathfrak{W}}:\ell^p\rightarrow\ell^p$ by, for any $x\in\ell^p$,
\[
\pi_n\circ F_{\mathfrak{W}}(x)=f_{w_n}(\pi_n(x)) \ \ \ \  \text{for any} \ n\in\mathbb{N}.
\]
Notice that $F_{\mathfrak{W}}$ is a norm preserving mapping. Then, by the continuity of each $f_{w_n}$, $F_{\mathfrak{W}}$ is a continuous mapping. In addition, define $F^{-1}_{\mathfrak{W}}:\ell^p\rightarrow\ell^p$ by, for any $x\in\ell^p$,
\[
\pi_n\circ F^{-1}_{\mathfrak{W}}(x)=f^{-1}_{w_n}(\pi_n(x)) \ \ \ \  \text{for any} \ n\in\mathbb{N}.
\]
Then, $F^{-1}_{\mathfrak{W}}$ is also a continuous mapping and is the inverse of $F_{\mathfrak{W}}$. Therefore, $F_{\mathfrak{W}}$ is a homeomorphism on $\ell^{p}$. Following from the equation (\ref{commu}), one can see that
\[
F_{\mathfrak{W}}\circ D_{\vert\mathfrak{W}\vert}=D_{\mathfrak{W}} \circ F_{\mathfrak{W}}
\]
Hence, $F_{\mathfrak{W}}$ is a topological conjugacy from $D_{\vert\mathfrak{W}\vert}$ to $D_{\mathfrak{W}}$.
\end{proof}

\begin{remark}
If $\vert w\vert=1$, the above conclusion may be false since every nontrivial rotation can not be topologically conjugate to the identity.
\end{remark}

\begin{theorem}\label{TCCD}
Let $D_{\mathfrak{W}}$ be a bounded diagonal operator on $\ell^{p}$ with the diagonal sequence $\mathfrak{W}=\{w_{n}\}_{n=1}^{\infty}$. If $\inf_{n}\vert w_n\vert>1$, then $D_{\mathfrak{W}}$ is topologically conjugate to $2\mathbf{I}$, where $\mathbf{I}$ means the identity operator.
\end{theorem}
\begin{proof}
Let $\rho=\inf_{n}\vert w_n\vert>1$. For any $n\in\mathbb{N}$, let $s_n=\log_{\rho}\vert w_n\vert$. Denote $\mathfrak{S}=\{s_{n}\}_{n=1}^{\infty}$. Since $D_{\mathfrak{W}}$ be a bounded diagonal operator, we have $s_n\geq1$ for every $n\in\mathbb{N}$. Recall that the map $h^{\mathfrak{S}}_{p}$ on $\ell^{p}$  is defined by
\[
\pi_{n}\circ h^{\mathfrak{S}}_{p}(x)=\frac{x_{n}}{\vert x_{n}\vert} \cdot \sqrt[p]{(\sum_{k=n}^{\infty}\vert x_{k}\vert^{p})^{s_{n}}-(\sum_{k=n+1}^{\infty}\vert x_{k}\vert^{p})^{s_{n}}},
\]
for any $x=(x_{1},x_{2},\ldots)\in \ell^{p}$.  Following from the Key Lemma, the map $h^{\mathfrak{S}}_{p}$ defined previously is a homeomorphism from $\ell^{p}$ onto itself. Notice that for any $x\in\ell^{p}$,
\[
h^{\mathfrak{S}}_{p}(\rho \cdot x)=D_{\vert\mathfrak{W}\vert}(h^{\mathfrak{S}}_{p}(x)).
\]
Then, $h^{\mathfrak{S}}_{p}$ is a topological conjugacy from $\rho\mathbf{I}$ to $D_{\vert\mathfrak{W}\vert}$. Furthermore, if we choose $\mathfrak{S}'$ be the sequence of constant $\log_{\rho}2$, it is not difficult to see that $h^{\mathfrak{S}'}_{p}$ is a topological conjugacy from $\rho\mathbf{I}$ to $2\mathbf{I}$. Therefore, by Lemma \ref{rotation}, $D_{\mathfrak{W}}$ is topologically conjugate to $2\mathbf{I}$.
\end{proof}

\begin{remark}
The above conclusion is obviously true for finite dimensional space $\mathbb{C}^m$. In fact,  we may choose a more simple topological conjagacy $h_m:\mathbb{C}^m\rightarrow\mathbb{C}^m$ defined by
\[
\pi_{n}\circ h_m(x)=\frac{x_{n}}{\vert x_{n}\vert} \cdot \vert x_{n}\vert^{s_n}, \ \  \ \text{for} \ \ n=1, \ldots, m.
\]
However, if we define $h$ by
\[
\pi_{n}\circ h(x)=\frac{x_{n}}{\vert x_{n}\vert} \cdot \vert x_{n}\vert^{s_n}, \ \  \ \text{for} \ \ n\in\mathbb{N},
\]
$h$ is not a continuous map on the infinite dimensional Banach space $\ell^p$. So, our Key Lemma plays an important role in the present paper.
\end{remark}

Similarly, we could also obtain that all of the bounded invertible diagonal operators with norm less than $1$ are topologically conjugate.

\begin{corollary}\label{TCCD01}
Let $D_{\mathfrak{W}}$ be a bounded diagonal operator on $\ell^{p}$ with the diagonal sequence $\mathfrak{W}=\{w_{n}\}_{n=1}^{\infty}$. If $\inf_{n}\vert w_n\vert>0$ and $\sup_{n}\vert w_n\vert<1$, then $D_{\mathfrak{W}}$ is topologically conjugate to $\frac{1}{2}\mathbf{I}$.
\end{corollary}

In the Theorem \ref{TCCD}, the condition $\inf_{n}\vert w_n\vert>1$ is necessary. We show the necessity as follows.

\begin{proposition}\label{NTC}
Let $D_{\mathfrak{W}}$ and $D_{\mathfrak{T}}$ be two bounded diagonal operators on $\ell^{p}$ with the diagonal sequence $\mathfrak{W}=\{w_{n}\}_{n=1}^{\infty}$ and $\mathfrak{T}=\{t_{n}\}_{n=1}^{\infty}$, respectively. If $\inf_{n}\vert w_n\vert=1$ and $\inf_{n}\vert t_n\vert>1$, then $D_{\mathfrak{W}}$ and $D_{\mathfrak{T}}$ are not topologically conjugate.
\end{proposition}

\begin{proof}
By Theorem \ref{TCCD}, it suffices to prove that $D_{\mathfrak{W}}$ and ${2}\mathbf{I}$ are not topologically conjugate.

Suppose that there exists $n\in\mathbb{N}$ such that $\vert w_n \vert=1$. Then $D_{\mathfrak{W}}$ has an invariant compact subset homeomorphic to the unit circle. Notice that
${2}\mathbf{I}$ has the unique invariant compact subset $\{0\}$. Thus, $D_{\mathfrak{W}}$ is not topologically conjugate to ${2}\mathbf{I}$ in this case.

Now suppose that $\vert w_n \vert>1$ for all $n\in\mathbb{N}$. Assume that $h$ is a topological conjugacy from $D_{\mathfrak{W}}$ to ${2}\mathbf{I}$. Then, $h(0)=0$ since $0$ is the unique fixed point. By the continuity of $h$, there is a positive number $\epsilon$ such that $h(\overline{B(0, 2\epsilon)})\subseteq B(0,1)$, where $B(0,2\epsilon)$ denotes the set $\{x\in \ell^{p}; ~\parallel x\parallel_p<2\epsilon\}$. Moreover, by the continuity of $h^{-1}$, there is a positive number $\delta$ such that $h^{-1}(\overline{B(0, \delta)})\subseteq B(0,\epsilon)$, i.e., $\overline{B(0, \delta)}\subseteq h(B(0,\epsilon))$. Since $\inf_{n}\vert w_n\vert=1$, there exists a subsequence $\{n_k\}_{k=1}^{\infty}$ such that $\vert w_{n_k}\vert^{k}\leq 2$. Let
\[
x^{(k)}=(0,0,\ldots,0,\underbrace{\epsilon}_{n_{k}-\text{th}},0,\ldots).
\]
Then,
\[
D_{\mathfrak{W}}^{k}(x^{(k)})=(0,0,\ldots,0, \underbrace{\vert w_{n_k}\vert^{k}\cdot\epsilon}_{n_{k}-\text{th}}, 0,\ldots)\in \overline{B(0, 2\epsilon)}.
\]
Since $h$ is a homeomorphism, one can see that
\[
h(x^{(k)})\in h(\overline{B(0, 2\epsilon)})\setminus \overline{B(0, \delta)}\subseteq B(0,1)\setminus \overline{B(0, \delta)}
\]
and
\[
h(D_{\mathfrak{W}}^{k}(x^{(k)}))\in h(\overline{B(0, 2\epsilon)})\setminus \overline{B(0, \delta)}\subseteq B(0,1).
\]
However,
\[
\parallel h(D_{\mathfrak{W}}^{k}(x^{(k)}))\parallel=\parallel 2^kh(x^{(k)})\parallel\geq 2^k\cdot\delta\rightarrow \infty \ \ \ \text{as} \ k\rightarrow\infty.
\]
This is a contradiction.
\end{proof}

\begin{example}
Let $\mathfrak{W}=\{w_{n}=1+1/n\}_{n=1}^{\infty}$. One can see that the dynamical properties of  diagonal operator $D_{\mathfrak{W}}$ and ${2}\mathbf{I}$ are almost the same. However, by Proposition \ref{NTC}, $D_{\mathfrak{W}}$ and ${2}\mathbf{I}$ are not topologically conjugate.
\end{example}

\section*{Declarations}

\noindent \textbf{Ethics approval}

\noindent Not applicable.

\noindent \textbf{Competing interests}

\noindent The author declares that there is no conflict of interest or competing interest.

\noindent \textbf{Authors' contributions}

\noindent All authors contributed equally to this work.

\noindent \textbf{Availability of data and materials}

\noindent Data sharing is not applicable to this article as no data sets were generated or analyzed during the current study.

\end{document}